\documentclass{article}

\usepackage{amsmath,xspace,comment,amssymb,amsthm,mathtools}
\newcommand\blfootnote[1]{%
  \begingroup
  \renewcommand\thefootnote{}\footnote{#1}%
  \addtocounter{footnote}{-1}%
  \endgroup
}

\newcommand{\bG}{\mathbb{G}}
\newcommand{\bH}{\mathbb{H}}
\newcommand{\ba}{\backslash}
\newcommand{\si}{\sigma}
\newcommand{\F}{\mathcal{F}}
\newcommand{\NP}{\text{\textit{NP}}_3}
\newcommand{\BR}{\textit{BR}}
\newtheorem{theorem}{Theorem}[section]
\newtheorem{corollary}[theorem]{Corollary}
\newtheorem{lemma}[theorem]{Lemma}
\newtheorem{proposition}[theorem]{Proposition}

\theoremstyle{definition}
\newtheorem{definition}[theorem]{Definition}
\newtheorem{remark}{Remark}

\title{Irreducibility of the Tutte polynomial of an embedded graph}

\author{Joanna A. Ellis-Monaghan\\
\small Korteweg-de Vries Instituut voor Wiskunde\\[-0.8ex]
\small Universiteit van Amsterdam, Science Park 105-107, 1098 XG\\[-0.8ex]
\small Amsterdam, The Netherlands\\
\small\tt j.a.ellismonaghan@uva.nl
\and
Andrew J. Goodall\\
\small Computer Science Institute (I\'UUK), Charles University\\[-0.8ex]
\small Malostransk\'e n\'am. 25,\\[-0.8ex]
\small 118 00 Praha 1, Czech Republic\\
\small\tt andrew@iuuk.mff.cuni.cz
\and
Iain Moffatt\\
\small Department of Mathematics \\[-0.8ex]
\small Royal Holloway, University of London\\[-0.8ex]
\small Egham, UK\\
\small\tt iain.moffatt@rhul.ac.uk
\and
Steven Noble\\
\small Birkbeck, University of London\\[-0.8ex]
\small Malet Street\\[-0.8ex]
\small London, UK\\
\small\tt steven.noble@bbk.ac.uk
\and
Llu\'{\i}s Vena\\
\small Universitat Polit\`ecnica de Catalunya\\[-0.8ex]
\small Barcelona\\[-0.8ex]
\small Spain\\
\small\tt lluis.vena@upc.edu}

\date{29th April 2022}

\begin{document}
\maketitle

\blfootnote{Acknowledgements: This work was undertaken while three of the authors (J.E-M, I.M and S.N.) were visiting the Department of Applied Mathematics (KAM) and the Computer Science Institute (I\'UUK), Charles University, Prague. They are grateful to KAM and I\'UUK for their generous hospitality and excellent research environment. J.E-M. was hosted and supported by KAM and by the Fulbright Commission as a Charles--Fulbright Distinguished Chair. A.G. was supported by Czech Science Foundation GA \v{C}R 19-21082S.  I.M. was supported by KAM and by a Scheme 4 grant from the London Mathematical Society. S.N. was  supported by KAM. L.V. is supported by Beatriu de Pin\'os BP2018, funded by the AGAUR (Government of Catalonia) and by the Horizon 2020 programme No 801370, and was supported by Czech Science Foundation GA \v{C}R 18-13685Y}

\begin{abstract}
We prove that the ribbon graph polynomial of a graph embedded in an orientable surface is irreducible if and only if the embedded graph is neither the disjoint union nor the join of embedded graphs. This result is analogous to the fact that the Tutte polynomial of a graph is irreducible if and only if the graph is connected and non-separable.

\smallskip
\noindent
\textbf{Mathematics Subject Classifications:} 05C31, 05B35

\smallskip
\noindent
\textbf{Keywords:} Bollob\'as--Riordan polynomial, delta-matroid, irreducible, ribbon graph, ribbon graph polynomial, separable, Tutte polynomial

\end{abstract}

\section{Introduction}
The \emph{Tutte polynomial} of a graph $G=(V,E)$ can be defined by
\begin{equation}\label{tsum}T(G;x,y) := \sum_{A\subseteq E} (x-1)^{r(E)-r(A)} (y-1)^{|A|-r(A)}, \end{equation}
where $r(A)$ is the rank of the subgraph $(V,A)$ of $G$.
It satisfies a universality property, which roughly means that it contains all graph parameters that satisfy a linear relation among $G$, $G\ba e$ and $G/e$  (see e.g.~\cite[Sec.~2.4]{handbook} for details). Because of this, the Tutte polynomial captures a surprisingly diverse range of graph parameters and appears in a variety of areas, such as statistical physics, knot theory, and coding theory. (See, for example, \cite{Bobook,BO,EMM,handbook,Welsh} for further background.)

 A standard property of the Tutte polynomial is that when a graph $G$ is the disjoint union or the one-point join of  graphs $G_1$ and $G_2$ we have
  $ T(G;x,y) =T(G_1;x,y)\, T(G_2;x,y)$.
  Thus if $T(G;x,y)$ is irreducible over $\mathbb{Z}[x,y]$, then $G$ must be connected and non-separable. In the 1970's Brylawski~\cite{Bry} conjectured that the converse also holds. This conjecture was verified in 2001  by Merino, de~Mier and Noy~\cite{Merino}. Thus
    $T(G;x,y)$ is irreducible over  $\mathbb{Z}[x,y]$ (or $\mathbb{C}[x,y]$) if and only if $G$ is a connected and non-separable graph.

 \smallskip
 An \emph{embedded graph}  (or equivalently a \emph{combinatorial map}, \emph{ribbon graph}, etc.) can be thought of as a graph drawn on a closed surface in such a way that its edges do not intersect (except at any common vertices), and such that its  faces are homeomorphic to discs.
 Our aim is to extend the above irreducibility result to the setting of embedded~graphs.

The analogue of the Tutte polynomial for an embedded graph $\bG$ is the  \emph{ribbon graph polynomial}, $R(\bG;x,y)$, which is a universal deletion-contraction invariant for  embedded graphs.
Its definition differs from~\eqref{tsum} by modifying the rank function so that it records some topological information about the embedding.
For an embedded graph $\bG=(V,E)$ and subset $A$ of $E$, we define $\gamma (A)$ to be the Euler genus of the embedded subgraph $(V,A)$, which coincides with the genus of a neighbourhood of $(V,A)$
in the surface, in the case that the surface is non-orientable, and twice its genus in the orientable case. Now let $\si(A):=r(A)+\tfrac{1}{2}\gamma (A)$, and then let
\begin{equation}\label{rgpoly}
R(\bG;x,y) := \sum_{A\subseteq E} (x-1)^{\si(E)-\si(A)} (y-1)^{|A|-\si(A)}.
\end{equation}
Although the ribbon graph and Tutte polynomials coincide for graphs embedded in the sphere, they do not agree in general.
 We note that the polynomial $R(\bG;x,y)$ is, up to a prefactor, a two-variable specialisation of the well known four-variable Bollob\'as--Riordan polynomial of~\cite{BR02}. However, as discussed in Remark~\ref{canon}, there are good reasons to work with $R(\bG;x,y)$ rather than Bollob\'as--Riordan polynomial or any of the more general topological Tutte polynomials in the literature.

 We say an embedded graph is a \emph{join}  if it  can be obtained from two embedded graphs via a connected summing operation that acts as follows. Choose a disc in each surface whose boundary intersects the graph in that surface at exactly a single non-isolated vertex.   Then identify the two discs so that the vertices on their boundaries are also identified, and then delete the interior of the identified~discs.

 A standard property  (see~\cite{BR02}) of the ribbon graph polynomial is that  if $\bG$ is either the disjoint union or join of $\bG_1$ and $\bG_2$, then  $R(\bG;x,y)=R(\bG_1;x,y) \, R(\bG_1;x,y)$. We prove here that the converse holds in the orientable case.
\begin{theorem}\label{main}
Let $\bG$ be a graph embedded in an orientable surface. Then $R(\bG;x,y)$ is irreducible over $\mathbb{Z}[x,y]$ (or $\mathbb{C}[x,y]$) if and only if $\bG$  is connected and not a join of two smaller embedded graphs.
\end{theorem}

 \medskip

Theorem~\ref{main} is an analogue of Merino, de~Mier and Noy's  result referred to above that  $T(G;x,y)$ is irreducible  if and only if $G$ is connected and non-separable.
As with many results for the classical Tutte polynomial, this irreducibility property is properly understood in terms of matroids, and  was shown in this more general setting. Merino et al. proved that for a matroid $M$ the polynomial $T(M;x,y)$ is irreducible if and only if $M$ is connected.  (The Tutte polynomial is extended to a matroid by taking $r$ in Equation~\eqref{tsum} to be the rank function of the matroid.) The graph result follows from the matroid one by considering  cycle matroids of graphs.

The situation for the ribbon graph polynomial is similar: many properties of the ribbon graph polynomial are properly understood in terms of delta-matroids. Delta-matroids generalise matroids, in essence, by relaxing the requirement that bases all have the same size, the analogue of bases being called \emph{feasible sets}.
It is well known that many properties of graphs are actually properties of matroids. Similarly, many properties of embedded graphs are in fact properties of delta-matroids.
 In particular,  the ribbon graph polynomial, connectivity and joins  can be understood in terms of delta-matroids (details are provided below), and  Theorem~\ref{main} is properly a result about delta-matroids:
  \begin{theorem}\label{main2}
Let $D$ be an even delta-matroid. Then $T(D;x,y)$ is irreducible over $\mathbb{Z}[x,y]$ (or $\mathbb{C}[x,y]$) if and only if $D$  is connected.
\end{theorem}
The orientably embedded graph of Theorem~\ref{main} is replaced in Theorem~\ref{main2} by an even delta-matroid, defined as one whose feasible sets all have size of the same parity; and  the ribbon graph polynomial $R(\bG;x,y)$ is replaced by $T(D;x,y)$, the Tutte polynomial of the delta-matroid $D$. The latter is a universal deletion-contraction invariant for delta-matroids (just as the classical Tutte polynomial is for matroids) and can be defined using a sum similar to that in Equation~\eqref{tsum}, replacing the rank function $r$ with the average of the rank functions of `minimum and maximum matroids' that arise from  a delta-matroid. See Section~\ref{sec3} for details.
  Similarly to the graphs and matroids case, Theorem~\ref{main} follows from Theorem~\ref{main2} by considering the delta-matroid of an embedded graph.

Here we are considering an analogue of the Tutte polynomial for embedded graphs.
 There are many extensions of the Tutte polynomial from graphs to other types of combinatorial object. Our main motivation in undertaking this work lies in uncovering what properties are innate to  graphs or matroids, and what properties extend or should extend to a wider class of objects. The significance of embedded graphs and delta-matroids in this context is that they provide an effective step in  moving away from the classical setting of graphs and matroids -- they are different but not too different. What is especially interesting about  Theorem~\ref{main} and~\ref{main2} is that very little of the argument depends upon the specific class of objects (graphs, matroids, embedded graphs, or delta-matroids) that we are working with. This hints at a larger, yet to be understood structure that would help explain the irreducibility of
  graph polynomials such as the Tutte polynomial.

\begin{remark}\label{canon}
Our interest here is in extensions of the Tutte polynomial to graphs that are cellularly embedded in  surfaces (the cellular condition means that the faces are homeomorphic to discs). It is not  obvious how the Tutte polynomial should be extended from graphs to embedded graphs, and many candidates  have been proposed~\cite{BR01,BR02,GLRV,GKRV,HM,KMT,Kr,LV,MS,MT,negami}.
It is natural to ask why we chose the ribbon graph polynomial~$R(\bG;x,y)$  as the analogue of the Tutte polynomial, rather than any of these other graph polynomials.

The Tutte polynomial of a graph satisfies a deletion-contraction recurrence that allows its expression in terms of its evaluations on trivial graphs. While all of the polynomials mentioned above have deletion-contraction relations that apply to particular types of edges of a cellularly embedded graph, only the ribbon graph polynomial has a  ``full'' deletion-contraction definition that applies to all edge-types.

 In more detail, there is a way to associate a ``canonical Tutte polynomial'' with a class of combinatorial objects~\cite{Fink,KMT}. The resulting polynomials are universal deletion-contraction invariants for that class, just as the classical Tutte polynomial is for the class of graphs.
 In this framework, the ribbon graph polynomial $R(\bG;x,y)$ arises as the  polynomial  associated with graphs that are cellularly embedded in surfaces, and hence is the universal deletion-contraction invariant for this class. (A similar comment holds for the delta-matroid version of the ribbon graph polynomial.) All of the other topological graph polynomials mentioned above arise in this framework as deletion-contraction invariants associated with other types of embedded graphs (for example, the Bollob\'as--Riordan polynomial arises as the  universal deletion-contraction invariant for graphs that are non-cellularly embedded in surfaces). See \cite{HM,KMT,  MS, MT} for details.
\end{remark}

\section{Background and notation}\label{bandn}
\subsection{Ribbon graphs}\label{embg}

It is convenient to realise embedded graphs as ribbon graphs. We give a brief overview of ribbon graphs, referring the reader to~\cite{EMMbook} or~\cite{GT87} (where they are called reduced band decompositions) for additional details, including their equivalence with (cellularly) embedded graphs.
 A ribbon graph is a structure that arises by taking a regular neighbourhood of a graph embedded in a surface while keeping the vertex--edge structure of the graph. Informally it can be thought of as ``a graph with vertices as discs and  edges as ribbons''. Formally, a \emph{ribbon graph} $\bG =\left(  V,E  \right)$ is a  surface with boundary represented as the union of two  sets of  discs, a set $V $ of {\em vertices}, and a set $E$ of {\em edges}  such that: (1)~the vertices and edges intersect in disjoint line segments;
(2)~each such line segment lies on the boundary of precisely one
vertex and precisely one edge;
(3)~every edge contains exactly two such line segments.

Graph-theoretic terminology naturally extends to ribbon graphs. A \emph{ribbon subgraph} $\bH$ of $\bG$ is a ribbon graph obtained from $\bG$ by removing some of its vertices and edges. It is \emph{spanning} if it has the same vertices as $\bG$.
The  \emph{rank} $r(\bG)$ of a ribbon graph $\bG=(V,E)$  is its number of vertices minus its number of connected components, that is,  it is the rank of its underlying graph. For $A\subseteq E$, $r(A)$ is the rank of the ribbon subgraph $(V,A)$ of $\bG$.

Topologically, a ribbon graph is a surface with boundary.
A \emph{quasi-tree} is a ribbon graph that has exactly one boundary component. A ribbon subgraph $\bH$ is a \emph{spanning quasi-tree} of a connected ribbon graph $\bG$ if it is a  quasi-tree that contains all the vertices of $\bG$. If $\bG$ is not connected, then we say $\bH$  is a spanning quasi-tree if for each connected component of $\bG$ the ribbon subgraph of $\bH$ obtained by removing vertices and edges not in this component is a spanning quasi-tree.

 A ribbon graph is \emph{orientable} if it is orientable when considered as a surface with boundary.   The \emph{genus} of a ribbon graph is its genus as a surface with boundary. The \emph{Euler genus}~$\gamma(\bG)$ of a ribbon graph~$\bG=(V,E)$  is its genus if it is non-orientable, and twice its genus if it is orientable. For $A\subseteq E$, $\gamma(A)$ is the Euler genus of the ribbon subgraph $(V,A)$ of $\bG$. A ribbon graph is \emph{plane} if it has Euler genus zero (note that we allow plane ribbon graphs to be disconnected).
The \emph{ribbon graph polynomial}  $R(\bG;x,y) $ of $\bG$ is defined as in Equation~\eqref{rgpoly}, where again $\si(A):=r(A)+\tfrac{1}{2}\gamma (A)$.

\subsection{Delta-matroids}\label{dm}
We shall work in the setting of delta-matroids and from this recover our results for embedded graphs.
We assume familiarity with the basic definitions of matroid theory~\cite{Oxley}, and give an overview of the  delta-matroid theory we use here. We refer the reader to~\cite{CMNR-JCT,Msurvey} for additional background on delta-matroids, which were introduced by Bouchet in~\cite{Bouchet:symm}. Equivalent concepts albeit using different terminology were also introduced at around the same time in~\cite{chand} and~\cite{dress}.

A \emph{delta-matroid} $D$ comprises a pair $(E,\mathcal F)$ where $E$ is a finite set and $\mathcal F$ is a non-empty collection of subsets of $E$ with the property that for all triples $(F_1,F_2,e)$ comprising members $F_1$ and $F_2$ of $\mathcal F$ and an element $e$ of $F_1 \bigtriangleup F_2$, there is an element $f$ of $F_1 \bigtriangleup F_2$ (which may be equal to $e$) such that $F_1 \bigtriangleup \{e,f\} \in \mathcal F$. This property is known as the \emph{symmetric exchange axiom}. The members of $\mathcal F$ are called \emph{feasible sets}, and $E$ is called its \emph{ground set}. It is not difficult to see that matroids are precisely delta-matroids in which the feasible sets are equicardinal.

Given a delta-matroid $D$, let $\mathcal F(D)$ denote its collection of feasible sets, and let $\mathcal F_{\max}$
and $\mathcal F_{\min}$ denote the subsets of $\mathcal F(D)$ comprising the feasible sets with maximum and minimum size respectively. It is straightforward to show that both $(E,\mathcal F_{\max})$ and $(E,\mathcal F_{\min})$ are matroids, known as the \emph{maximum} and \emph{minimum} matroids and denoted by $D_{\max}$ and $D_{\min}$ respectively.

For a matroid $M$, let $r(M)$ denote its rank and let $r_M(A)$ denote the rank of the set $A$ of elements of $M$.
For a delta-matroid $D$ with element set $E$ and set $\mathcal F$ of feasible sets,  the \emph{delta-matroid rank function}, $\rho_D$, introduced by Bouchet in~\cite{Bouchet:rep} is given by
\[ \rho_D(A)=|E| - \min\{|A \bigtriangleup F|: F\in \mathcal F\}.\]
Note that if a delta-matroid $D$ is also a matroid, then $\rho_D$ and $r_D$ do not generally coincide. This explains why we do not define the Tutte polynomial of a delta-matroid by merely replacing $r$ by $\rho$ in Equation~\eqref{tsum}.

A \emph{coloop} of $D$ is an element of $D$ belonging to every feasible set. A \emph{loop} of $D$ is an element of $D$ belonging to no feasible set.

Let $D$ be a delta-matroid and $e$ an element of $D$. Suppose first that $e$ is not a coloop of $D$. Then we define $D\ba e$, the \emph{deletion} of $e$, to be the pair
\[ (E-e,\{ F \in \mathcal {F} \mid e\notin F\}).\]
Now suppose that $e$ is not a loop of $D$. Then we define
$D/ e$, the \emph{contraction} of $e$, to be the pair
\[ (E-e,\{ F-e\mid F \in \mathcal {F} \text{ and } e \in F\}).\]
If $e$ is either a coloop or a loop of $D$, then one of $D\ba e$ and $D/e$ is defined. In this case, we define whichever of $D\ba e$ and $D/e$ is so far undefined by setting $D\ba e = D/e$. It is easy to check that both $D\ba e$ and $D/e$ are delta-matroids. Moreover it is also easy to check that if we perform a sequence of deletions and contractions then the resulting delta-matroid does not depend on the order in which these operations are carried out. Thus we may delete and contract sets of elements without ambiguity. Any delta-matroid obtained from $D$ by deleting and contracting possibly empty subsets of the elements of $D$ is said to be a \emph{minor} of $D$.

For a subset $A$ of the element set of $D$, let $D|A=D\ba A^c$ denote the delta-matroid formed by deleting the elements of $A^c:=E\setminus A$ and let $\sigma(A)=(r((D|A)_{\max}) + r((D|A)_{\min}))/2$.
The \emph{width} $w(D)$ of $D$ is $r(D_{\max}) - r(D_{\min})$.
Note that $\sigma(A)= r((D|A)_{\min}) +w(D|A)/2$.

Just as the spanning trees in a graph give rise to its cycle matroid, the spanning quasi-trees in a ribbon graph give rise to its delta-matroid.
For a ribbon graph $\bG=(V,E)$, the pair $D(\bG):=(E,\F)$, where
\[\F:=\{F\subseteq E : F\text{ is the edge set of a  spanning quasi-tree of }\bG\},\]
is the \emph{delta-matroid of $\bG$}.
These delta-matroids can be regarded as the topological analogues of the cycle matroids of graphs.
A delta-matroid arising from a ribbon graph in this way is said to be \emph{ribbon-graphic}.
The class of ribbon-graphic delta-matroids was first considered by Bouchet in~\cite{Bouchet:maps}, albeit using very different language.
In Proposition~5.3 of~\cite{CMNR-JCT}, it is shown that $w(D(\bG))=\gamma(\bG)$, and consequently $\sigma(D(\bG))=\si(\bG)$.

\section{The Tutte polynomial of a delta-matroid}\label{sec3}

We begin by extending the definition of the Tutte polynomial of a matroid to delta-matroids.
For a delta-matroid $D$ with element set $E$, define its \emph{Tutte polynomial}  $T(D;x,y)$ by
\begin{equation}\label{edef} T(D;x,y) := \sum_{A\subseteq E} (x-1)^{\sigma(E)-\sigma(A)}(y-1)^{|A|-\sigma(A)}.\end{equation}
Note that if $D$ is a matroid, then for every subset $A$ of its elements, $r((D|A)_{\min})=r((D|A)_{\max})$, so $\sigma(A)=r(A)$. Therefore our definition of the Tutte polynomial of a delta-matroid is consistent with the existing definition of the Tutte polynomial of a matroid and retains several key properties.

Following~\cite{CMNR-JCT}, the \emph{Bollob\'as--Riordan polynomial} of a delta-matroid $D$ is given by
\[ \BR(D;x,y,z):= \sum_{A\subseteq E}
(x-1)^{r_{D_{\min}}(E)-r_{D_{\min}}(A)}
y^{|A|-r_{D_{\min}}(A)}
z^{w(D|A)}.\]

Since $\sigma(D(\bG))=\si(\bG)$ for any ribbon graph $\bG$, the ribbon graph polynomial of $\bG$ agrees with the Tutte polynomial of its delta-matroid:   $R(\bG;x,y)=T(D(\bG);x,y)$.
Similarly, the Bollob\'as--Riordan polynomial of a ribbon graph, introduced  in~\cite{BR02}, agrees with the  Bollob\'as--Riordan polynomial of its delta-matroid (see Theorem~6.4 of~\cite{CMNR-JCT}).

\medskip
The next two results are from~\cite{CMNR-JCT}. The first is stated on page~52 and the second is Theorem~6.6(1).
\begin{lemma}\label{lem:RGandBR}
For every delta-matroid $D$,
\[T(D;x,y)= (x-1)^{w(D)/2}\BR(D;x,y-1,1/\sqrt{(x-1)(y-1)}\:).\]
\end{lemma}

\begin{proposition}\label{prop:previousresult}
For every delta-matroid $D$ with element set $E$,
\[ v^{\sigma(D)}u^{-w(D)/2}T(D;u/v+1,uv+1)=\sum_{A\subseteq E}
v^{|A|}u^{|E|-\rho_D(A)}.\]
\end{proposition}

Recall that a delta-matroid is \emph{even} if and only if the cardinalities of its feasible sets all have the same parity. The property of being even is preserved under deletion and~contraction.

\begin{corollary}~\label{cor:basicparams}
For every delta-matroid $D$, the polynomial $T(D;x,y)$ determines the following:
\begin{enumerate}
\item the number of elements of $D$;
\item the number of feasible sets in $D$ of given size;
\item the ranks of the minimum and maximum matroids of $D$;
\item the width
of $D$;
\item whether or not $D$ is even;
\item whether or not $D$ is a matroid; and
\item in the case where $D$ is the delta-matroid of a ribbon graph $\bG$,  whether or not $\bG$ is plane.
\end{enumerate}
\end{corollary}

\begin{proof}
It follows from Proposition~\ref{prop:previousresult} that
the minimum degree of $v$ in $T(D;u/v+1,uv+1)$ is $-\sigma(D)$ and
the maximum degree of $v$ in $T(D;u/v+1,uv+1)$ is $|E(D)|-\sigma(D)$.
Thus both $|E(D)|$ and $\sigma(D)$ are determined by~$T(D)$.

As $A$ is feasible in $D$ if and only if $\rho_D(A)=|E|$, the terms of
$T(D;u/v+1,uv+1)$ with minimum degree in $u$ correspond to the feasible sets of $D$. Such a set $F$ yields a term $u^{w(D)/2}v^{|F|-\sigma(D)}$, so one may deduce the number of feasible sets of $D$ of any given size. In particular, $T(D)$ determines the ranks of the minimum and maximum matroids of the delta-matroid $D$ and consequently $w(D)$, and whether or not $D$ is even. As $T(D)$ determines the width of $D$, it also determines whether or not $D$ is a matroid. If $D$ is the ribbon-graphic delta-matroid of a ribbon graph $\bG$, then, since $w(D(\bG))=\gamma(\bG)$, $D$ is a matroid if and only if $\bG$ is plane.
\end{proof}

Given delta-matroids $D_1=(E_1,\mathcal F_1)$ and $D_2=(E_2,\mathcal F_2)$ with disjoint element sets, let $D_1 \oplus D_2$ denote the delta-matroid with element set $E_1 \cup E_2$ and set of feasible sets $\{F_1\cup F_2: F_1\in\mathcal F_1 \text{ and } F_2\in \mathcal F_2\}$. We say that a delta-matroid $D$ is \emph{connected} if there do not exist delta-matroids $D_1$ and $D_2$ with non-empty element sets satisfying $D=D_1 \oplus D_2$.

\begin{proposition}\label{prop:directsum}
Let $D$ be a delta-matroid $D$ such that $D=D_1\oplus D_2$. Then
\[T(D;x,y)=T(D_1;x,y)\,T(D_2;x,y).\]
\end{proposition}
\begin{proof}
Let $A_1$ and $A_2$ be subsets of the element sets of $D_1$ and $D_2$, and let $A=A_1\cup A_2$. Then $D|A=(D_1|A_1)\oplus (D_2|A_2)$. Moreover for any delta-matroids $D'_1$ and $D'_2$ on disjoint element sets $(D'_1 \oplus D'_2)_{\min} = (D'_1)_{\min} \oplus (D'_2)_{\min}$ and
$(D'_1 \oplus D'_2)_{\max} = (D'_1)_{\max} \oplus (D'_2)_{\max}$, so all the relevant parameters are additive over the components $D_1$ and $D_2$ of $D$. Hence the result follows.
\end{proof}

Following Bouchet~\cite{Bouchet:symm}, for a delta-matroid $D$ and subset $A$ of its elements, let $D*A$ denote the \emph{twist} of $D$ with respect to $A$, that is,
the delta-matroid with the same element set as $D$ such that $F$ is feasible in $D*A$ if and only if $F\bigtriangleup A$ is feasible in $D$. The \emph{dual} $D^*$ of $D=(E,\mathcal F)$ is $D*E$. Observe that $w(D)=w(D^*)$, $D^*/e=(D\ba e)^*$, and $D^*\ba e=(D/ e)^*$.
Recall that a \emph{loop} of $D$ is an element that appears in no feasible set of $D$.
An element $e$ of a delta-matroid $D$ is:
\begin{enumerate}
\item \emph{ordinary} if $e$ is not a loop in $D_{\min}$;
\item \emph{a non-orientable ribbon loop} if $e$ is a loop in both $D_{\min}$ and $(D*e)_{\min}$;
\item \emph{an orientable ribbon loop} if $e$ is a loop in $D_{\min}$ but not in $(D*e)_{\min}$.
\end{enumerate}
Clearly, every element of $D$ is of exactly one of these three types. Moreover, it is not difficult to see that in an even delta-matroid every element is either ordinary or an orientable ribbon loop. Note that every loop is an orientable ribbon loop, but the converse does not generally hold.

The next result follows from Lemma~\ref{lem:RGandBR} and the deletion-contraction recurrences for~$\BR(D)$ (Corollary~5.10 of~\cite{CMNR-LMS}). It is possible to add extra cases corresponding to non-orientable ribbon loops of $D$ or $D^*$, but we omit these as we will not require them.

\begin{proposition}\label{prop:delcon}
For every delta-matroid $D$, the following hold.
\begin{enumerate}
    \item\label{prop:delcon1} If the ground set of $D$ is empty, then $T(D;x,y)=1$.
    \item\label{prop:delcon2} If element $e$ is ordinary in both $D$ and $D^*$, then
    \[ T(D;x,y) = T(D\ba e;x,y) + T(D/e;x,y).\]

    \item\label{prop:delcon3} If element $e$ is ordinary in $D$ and an orientable ribbon loop in $D^*$, then
    \[ T(D;x,y) = (x-1)T(D\ba e;x,y) + T(D/e;x,y).\]

    \item\label{prop:delcon4} If element $e$ is an orientable ribbon loop in $D$ and ordinary in $D^*$, then
    \[ T(D;x,y) = T(D\ba e;x,y) + (y-1)T(D/e;x,y).\]

    \item\label{prop:delcon5} If element $e$ is an orientable ribbon loop in both $D$ and $D^*$, then
    \[ T(D;x,y) = (x-1)T(D\ba e;x,y) + (y-1)T(D/e;x,y). \]
\end{enumerate}
\end{proposition}

It is well known that the Tutte polynomials of a matroid $M$ and its dual satisfy $T(M;x,y)=T(M^*;y,x)$. This relation extends to delta-matroids, as shown in Theorem~6.6 of~\cite{CMNR-JCT}.
\begin{proposition}\label{prop:tuttedual}
For every delta-matroid $D$, $T(D;x,y)=T(D^*;y,x)$.
\end{proposition}
When $D$ is even, this can be proved using Proposition~\ref{prop:delcon} and   induction.

Let $D$ be a delta-matroid.
Recall that a~\emph{coloop} of $D$ is an element appearing in every feasible set of~$D$. Clearly an element $e$ of $D$ is a loop if and only if it is a coloop of~$D^*$.

\begin{corollary}\label{cor:loopscoloops}
For every delta-matroid~$D$, the following hold.
\begin{enumerate}
    \item If element $e$ is a coloop of $D$, then $T(D;x,y)=xT(D/e;x,y)=xT(D\ba e;x,y)$.
    \item If element $e$ is a loop of $D$, then $T(D;x,y)=yT(D/e;x,y)=yT(D\ba e;x,y)$.
        \end{enumerate}
\end{corollary}

One of the many well known evaluations of the Tutte polynomial of a matroid $M$ is that $T(M;1,1)$ is equal to the number of bases of $M$. As shown in~\cite{CMNR-JCT},  this evaluation only extends to $T(D)$ when $D$ is a matroid.
\begin{proposition}
\label{prop:countbases}
For every delta-matroid $D$,
\[ T(D;1,1) = \begin{cases} \text{the number of bases of $D$} & \text{if $D$ is a matroid,} \\ 0 & \text{otherwise.}\end{cases}\]
\end{proposition}

We easily obtain the following corollaries.
\begin{corollary}\label{cor:countbases1}
Let $D$ be a delta-matroid that is not a matroid, and let $M$ be a matroid. Then $T(D;x,y)\ne T(M;x,y)$.
\end{corollary}

\begin{corollary}\label{cor:countbases2}
Let $\bG$ be a non-plane ribbon graph, and let $H$ be a graph. Then
$R(\bG;x,y)\ne T(H;x,y)$.
\end{corollary}

\begin{corollary}\label{cor:countbases3}
Let $\bG$ be a ribbon graph and $H$ a graph such that  $R(\bG;x,y)=T(H;x,y)$. Then $\bG$ is plane.
\end{corollary}

\section{The Beta Invariant}
The beta invariant of a matroid was introduced by Crapo in~\cite{Crapo} and encapsulates a surprisingly large amount of information. It was first extended to delta-matroids in~\cite{MMN}, where the focus was on the transition polynomial rather than the Tutte polynomial. Here we define the beta invariant of a  delta-matroid as follows.
\begin{definition}\label{defbeta}
The \emph{beta invariant}, $\beta(D)$, of a delta-matroid $D$ is the coefficient of $x$ in $T(D;x,y)$.
\end{definition}

For the special case of matroids, this coincides with Crapo's definition from~\cite{Crapo}. This is not the case for the invariant introduced in~\cite{MMN} and defined in terms of the transition polynomial, but the invariant described here is easily computed from the one introduced in~\cite{MMN} and vice versa. Versions of Theorems~\ref{thm:betaconn} and~\ref{thm:serpar} below
appear in~\cite{MMN}, phrased in terms of the transition polynomial.
For completeness, we include proofs which do not depend on any results from~\cite{MMN}.

\medskip

To deduce properties of $\beta(D)$, we need the following results.
The first follows immediately from the definition of the Tutte polynomial given in Equation~\eqref{edef}.
\begin{lemma}\label{lem:constant}
For every delta-matroid $D$ with element set $E$,
\[ T(D;0,0) = \sum_{A\subseteq E} (-1)^{\sigma(E)-|A|} = \begin{cases} 0 &\text{if $E\ne\emptyset$,} \\ 1 &\text{if $E=\emptyset$.}\end{cases}\]
\end{lemma}

\begin{proposition}\label{prop:x=y}
Let $D$ be an even delta-matroid with at least two elements. Then the coefficients of $x$ and of $y$ in $T(D;x,y)$ are equal.
\end{proposition}

\begin{proof}
We proceed by induction on the number of elements of $D$.
We begin by considering all the possibilities for $D$ when it has two elements.
If $D$ has two elements and is disconnected, then each element is either a loop or a coloop, so by Corollary~\ref{cor:loopscoloops} both the coefficients of $x$ and $y$ are zero. If $D$ has two elements and is connected, then it is equal to $D(\bG)$ where $\bG$ is either the plane cycle with two edges or the genus one orientable ribbon graph with one vertex and two edges. In the former case $T(D;x,y)=x+y$ and in the latter case $T(D;x,y)= 2xy-x-y$. So the result holds when $D$ has two elements.

The inductive step follows by combining Proposition~\ref{prop:delcon} with Lemma~\ref{lem:constant}.
\end{proof}

The following technical lemma is needed in most of the results of this section. An orientable ribbon loop that is not a loop is said to be a \emph{non-trivial orientable ribbon loop}.

\begin{lemma}\label{lem:widths}
For an even delta-matroid $D$ and element $e$ of $D$, the following hold.
\begin{enumerate}
    \item If $e$ is a loop or coloop of $D$, then $w(D\ba e)=w(D/e)=w(D)$.
    \item If $e$ is ordinary in both $D$ and $D^*$, then $w(D\ba e)=w(D/e)=w(D)$.
    \item If $e$ is a non-trivial orientable ribbon loop in $D$ and ordinary in $D^*$, then $w(D\ba e)=w(D)$ and $w(D/e)=w(D)-2$.
    \item If $e$ is ordinary in $D$ and a non-trivial orientable ribbon loop in $D^*$, then $w(D\ba e)=w(D)-2$ and $w(D/ e)=w(D)$.
    \item If $e$ is an orientable ribbon loop in both $D$ and $D^*$, then
    $w(D\ba e)=w(D/e)=w(D)-2$.
\end{enumerate}
\end{lemma}

\begin{proof}
If $e$ is a loop of $D$, then $\mathcal F(D)=\mathcal F(D\ba e) = \mathcal F(D/e)$. Therefore $w(D\ba e)=w(D/e)=w(D)$. By duality, the same conclusion holds if $e$ is a coloop of $D$. This proves~(1).

From now on we assume that $e$ is neither a loop nor a coloop of $D$. It follows from the symmetric exchange axiom that a coloop of $D_{\min}$ is a coloop of $D$. Thus $e$ is not a coloop of $D_{\min}$.
 Hence $r((D\ba e)_{\min})=r(D_{\min})$.
If $e$ is ordinary in $D$, then $r((D/e)_{\min}) =r(D_{\min})-1$. Now suppose that $e$ is a non-trivial orientable ribbon loop in $D$. Then as $e$ is not a loop, there is a feasible set containing $e$. Let $F_2$ be such a feasible set of minimum possible size. Let $F_1$ be a basis of $D_{\min}$. Applying the symmetric exchange axiom to $F_1$, $F_2$ and $e$, we deduce that there is a feasible set of $D$ containing $e$ and having size at most $r(D_{\min})+2$. As $D$ is even and $e$ is an orientable ribbon loop, such a set must have size $r(D_{\min})+2$. Thus $r((D/e)_{\min}) =r(D_{\min})+1$. Using duality, we see that $r((D/ e)_{\max})=r(D_{\max})-1$; if $e$ is ordinary in $D^*$ then $r(D\ba e)_{\max} =r(D_{\max})$, and if $e$ is a non-trivial orientable ribbon loop in $D^*$ then $r((D\ba e)_{\max}) =r(D_{\max})-2$.

Each of the remaining parts of the result now follows by applying the definition of width.
\end{proof}

\begin{proposition}\label{prop:keysignsstuff}
Let $D$ be an even delta-matroid with at least two elements.
Then $\beta(D)$ is either zero or has the same sign as $(-1)^{w(D)/2}$. Moreover, if $D$ has at least three elements and $e$ is an element of $D$ that is neither a loop nor a coloop, then the coefficient of $x$ in both terms appearing on the right side of the equation in each of parts (\ref{prop:delcon2})--(\ref{prop:delcon5}) of Proposition~\ref{prop:delcon} is either zero or has the same sign as $(-1)^{w(D)/2}$.
\end{proposition}

\begin{proof}
We proceed by induction on the number of elements of $D$. In the proof of Proposition~\ref{prop:x=y}, we showed that if $D$ has two elements then either  $\beta(D)=0$, $D$ has width zero and $\beta(D)=1$, or $D$ has width two and $\beta(D)=-1$. So the result holds in this case.

Now suppose that $D$ has at least three elements and $e$ is an element of $D$. If $e$ is either a loop or a coloop, then by Corollary~\ref{cor:loopscoloops} and Lemma~\ref{lem:constant}, $\beta(D)=0$. From now on, we assume that $e$ is neither a loop nor a coloop. We first prove the assertion about the coefficient of $x$ in any term appearing on the right side of the equation in a part of Proposition~\ref{prop:delcon}.

If $e$ is ordinary in both $D$ and $D^*$, then by Lemma~\ref{lem:widths}, $w(D)=w(D\ba e)=w(D/e)$, so the result follows from the inductive hypothesis.

If $e$ is an orientable ribbon loop in $D$ and ordinary in $D^*$, then \[ T(D;x,y) = T(D\ba e;x,y) + (y-1)T(D/e;x,y). \]
By Lemma~\ref{lem:widths}, $w(D\ba e)=w(D)$ and $w(D/e)=w(D)-2$. As $D\ba e$ and $D/e$ have at least two elements, Lemma~\ref{lem:constant} and the inductive hypothesis imply that the coefficient of $x$ in $T(D\ba e;x,y)$ is either zero or has the same sign as $(-1)^{w(D\ba e)/2}=(-1)^{w(D)/2}$
and the coefficient of $x$ in $(y-1)T(D/e;x,y)$ is either zero or has the same sign as
 $-(-1)^{w(D/e)/2} = (-1)^{w(D)/2}$, so the result follows from the inductive hypothesis.

The case where $e$ is ordinary in $D$ and an orientable ribbon loop in $D^*$ follows either by a very similar argument to the previous case or by applying the previous case to $D^*$ and using Propositions~\ref{prop:x=y} and~\ref{prop:tuttedual}.

Finally, if $e$ is an orientable ribbon loop in both $D$ and $D^*$, then
\[ T(D;x,y) = (x-1)T(D\ba e;x,y) + (y-1)T(D/e;x,y). \]
By Lemma~\ref{lem:widths}, $w(D\ba e)=w(D/e)=w(D)-2$. As both $D\ba e$ and $D/e$ have at least two elements, Lemma~\ref{lem:constant} and the inductive hypothesis imply that the coefficient of $x$ in each of $(x-1)T(D\ba e;x,y)$ and $(y-1)T(D/e;x,y)$ is either zero or has the same sign as $-(-1)^{w(D/e)/2} = (-1)^{w(D)/2}$, so the result follows from the inductive hypothesis.

By using Proposition~\ref{prop:delcon}, it follows immediately that
$\beta(D)$ is either zero or has the same sign as $(-1)^{w(D)/2}$. Hence the result follows by induction.\end{proof}

Bouchet proved the following theorem in the more general context of tight multimatroids. In a sense that is made precise in~\cite{Bouchet:MM3}, even delta-matroids are equivalent to a subclass of tight multimatroids.
\begin{theorem}
\label{thm:Bouchetchain}
Let $D$ be a connected even delta-matroid with element $e$. Then at least one of $D\ba e$ and $D/e$ is connected.
\end{theorem}

We now obtain the following property of $\beta$ generalizing a result of Crapo~\cite{Crapo} for matroids.
\begin{theorem}\label{thm:betaconn}
Let $D$ be an even delta-matroid with at least two elements. Then
$\beta(D) \ne 0$ if and only if $D$ is connected. Moreover if $D$ is connected, then the sign of $\beta(D)$ is the same as that of $(-1)^{w(D)/2}$.
\end{theorem}

\begin{proof}
If $D=D_1\oplus D_2$ then $T(D;x,y)=T(D_1;x,y)\,T(D_2;x,y)$, so by Lemma~\ref{lem:constant} if $D$ is disconnected then $\beta(D)=0$.

To prove the converse, we proceed by induction on the number of elements of $D$. If $D$ has two elements and is connected, then from the proof of Proposition~\ref{prop:x=y} we see that
$T(D;x,y)=x+y$ or $T(D;x,y)= 2xy-x-y$, so the result holds when $D$ has two elements.

Now suppose that the result holds for all even delta-matroids with fewer than $n$ elements. Let $D$ be a connected, even delta-matroid having $n>2$ elements.

As $D$ is connected, it has no loop or coloop. Let $e$ be an element of $D$.
By Theorem~\ref{thm:Bouchetchain} at least one of $D\ba e$ and $D/e$ is connected. Then the induction hypothesis implies that
at least one of $\beta(D\ba e)$ and $\beta(D/e)$ is non-zero. Combining this observation with Propositions~\ref{prop:delcon} and~\ref{prop:keysignsstuff}, we deduce that $\beta(D)\ne 0$. Applying Proposition~\ref{prop:keysignsstuff} again, we see that the sign of $\beta(D)$ is the same as that of $(-1)^{w(D)/2}$. Hence the result follows by induction.
\end{proof}

\begin{corollary}
For an even delta-matroid $D$, the following hold.
\begin{enumerate}
    \item $T(D;x,y)$ is divisible by $x^i$ if and only if $D$ has at least $i$ coloops.
    \item $T(D;x,y)$ is divisible by $y^j$ if and only if $D$ has at least $j$ loops.
\end{enumerate}
\end{corollary}

\begin{proof}
If $D$ has at least $i$ coloops, then Corollary~\ref{cor:loopscoloops} implies that $T(D;x,y)$ is divisible by $x^i$. Similarly if $D$ has at least $j$ loops, then $T(D;x,y)$ is divisible by $y^j$.

Now suppose that $D$ has $i$ loops, $j$ coloops and $k$ components $D_1,\ldots,D_k$ that are neither loops nor coloops. Then
\[ T(D;x,y) = x^i\,y^j \,T(D_1;x,y) \cdots T(D_k;x,y).\]
For each $l$, $D_l$ has at least two elements, so Proposition~\ref{prop:x=y} and Theorem~\ref{thm:betaconn} imply that $T(D_l;x,y) = a_l(x+y) + p_l(x,y)$, where $a_l$ is a non-zero constant and $p_l$ is a polynomial in which every monomial has degree at least two. Thus none of $T(D_1), \ldots, T(D_k)$ is divisible by either $x$ or $y$. Hence the result follows.
\end{proof}

We say that a delta-matroid is \emph{series-parallel} if there is a plane 2-connected series-parallel network $\bG$
such that $D$ is a twist of $D(\bG)$. Series-parallel delta-matroids were introduced in~\cite{MMN}, where
twisted duals of $D(\bG)$ were also considered.

The next result is a reformulation of a result from~\cite{MMN}. We include a  proof for completeness. We recall that $U_{2,4}$ and $M(K_4)$ denote, respectively, the matroid with four elements in which a set is independent if and only if it has size at most two and the cycle matroid of the complete graph with four vertices. We use $\NP$ to denote the delta-matroid on three elements in which a set is feasible if and only if it has even~size.

\begin{theorem}\label{thm:serpar}
Let $D$ be an even delta-matroid with at least two elements.
Then the following are equivalent.
\begin{enumerate}
    \item $D$ is series-parallel.
    \item $\beta(D)=(-1)^{w(D)/2}$.
\end{enumerate}
\end{theorem}

\begin{proof}
Suppose the first condition holds. Then there is a plane 2-connected series-parallel network $\bG$ such that $D$ is a twist of $D(\bG)$.
We proceed by induction on the number of edges in $\bG$.
If $\bG$ has only two edges, then
$\bG$ is a cycle with two edges and
the result is easy to check. Otherwise $\bG$ has an edge $e$ such that one of $\bG\ba e$ and $\bG/e$ contains a coloop or a loop, and the other is a plane 2-connected series-parallel network with at least two edges. Thus one of $D\ba e$ or $D/e$ contains a loop or a coloop, and the other is series-parallel.
So one of $\beta(D\ba e)$ and $\beta(D/e)$ is zero, and the inductive hypothesis together with Proposition~\ref{prop:delcon} and Lemma~\ref{lem:constant} imply that the other is~$\pm 1$.
Hence $\beta(D)=\pm 1$, and Theorem~\ref{thm:betaconn} determines the sign to be as stated in the second~condition.

Now suppose that $\beta(D)=(-1)^{w(D)/2}$.
 Suppose for a contradiction that $D$ is not series-parallel. Then either $D$ is not the twist of a matroid or it is the twist of a matroid $M$ that is not series-parallel. In the former case, by Proposition~1.5 of~\cite{Duchamp}, $D$ contains a twist of $\NP$ as a minor. In the latter case, by, for example, Corollary~12.2.14 of~\cite{Oxley}, $M$ contains $U_{2,4}$ or $M(K_4)$ as a minor, so $D$ contains a twist of either $U_{2,4}$ or $M(K_4)$ as a minor. Thus we see that $D$ contains a twist of either $\NP$, $U_{2,4}$ or $M(K_4)$ as a minor.
 Let $H$ denote such a minor.
As $D$ is connected (by Theorem~\ref{thm:betaconn}), we may apply Corollary~3.3 of~\cite{CCN:splitter} to find a sequence $D=D_0, D_1, \ldots, D_{k-1}, D_k=H$ of connected minors of $D$ such that for each $i$, $D_i$ is obtained from $D_{i-1}$ by deleting or contracting a single element $e_i$.
We calculate $T(D)$ by deleting or contracting each element $e_i$ in increasing order of $i$ and using Theorem~\ref{prop:delcon}.
As each of $D_0,D_1\ldots,D_{k-1}$ is connected, none of $e_1,e_2,\ldots,e_k$ is a loop or coloop.
Therefore we may apply Proposition~\ref{prop:keysignsstuff} to deduce that the contributions to $\beta$ from the two terms on the right side of the appropriate deletion-contraction relations never have opposing signs. Thus $|\beta(D)| \geq |\beta(H)|$.
It is easy to check that if $H$ is a twist of $U_{2,4}$, $\NP$ or $M(K_4)$, and $e$ is an element of $H$, then both $H\ba e$ and $H/e$ are connected. Hence by Proposition~\ref{prop:keysignsstuff} and Theorem~\ref{thm:betaconn}, $|\beta(H)| \geq 2$. So $|\beta(D)|\geq 2$, and the result follows.
\end{proof}

\section{Irreducibility of $T(D)$}

In this section we prove our two main results, which characterise when $R(\bG)$ and $T(D)$ are irreducible. Our approach adapts the argument from~\cite{Merino}.

For a delta-matroid $D$, we let $b_{i,j}(D)$ denote the coefficient of $x^iy^j$ in $T(D;x,y)$. When the context is clear, we just write $b_{i,j}$. By the duality formula of Proposition~\ref{prop:tuttedual},  for every delta-matroid $D$ we have $b_{i,j}(D^*)=b_{j,i}(D)$.

Brylawski~\cite{Bry} established a collection of affine relations satisfied by the coefficients of the Tutte polynomial of a matroid. Much later, in a surprising result, Gordon~\cite{Gordon} demonstrated that these affine relations hold much more generally.

For a delta-matroid with element set $E$, we have $\sigma(\emptyset)=0$, and for any subset $A$ of $E$ we have $\sigma(A) \leq \max \{|A|,\sigma(E)\}$. These conditions are sufficient to apply Theorem~11 of~\cite{Gordon} to show that all of Brylawski's affine relations hold for $T(D)$, giving the following.

\begin{theorem}\label{thm:Bry}
Let $D$ be an even delta-matroid with element set $E$ and let $n=|E|$. Then,
for all $k$ with $0\leq k < n$,
    \[ \sum_{i=0}^k \sum_{j=0}^{k-i} (-1)^j \binom{k-i}j b_{i,j}=0,\]
and
\[ \sum_{i=0}^n \sum_{j=0}^{n-i} (-1)^j \binom{n-i}j b_{i,j}=(-1)^{n-\sigma(E)}.\]
\end{theorem}

\begin{lemma}\label{lem:topcoeffsR}
Let $D$ be an even delta-matroid with element set $E$. If $i>\sigma(E)$ or $j>|E|-\sigma(E)$, then $b_{i,j}=0$. Moreover,
\[ \sum_{j=0}^{|E|-\sigma(E)} b_{\sigma(E),j} =
\sum_{i=0}^{\sigma(E)} b_{i,|E|-\sigma(E)} = 1.\]
\end{lemma}

\begin{proof}
The first part follows from the definition of $T$ and the observation preceding Theorem~\ref{thm:Bry}. Notice that $\sum_{j=0}^{|E|-\sigma(E)} b_{\sigma(E),j}$ is equal to the coefficient of $x^{\sigma(E)}$ in
$T(D;x,1)$. Thus, as $\sigma(\emptyset)=0$,
\[\sum_{j=0}^{|E|-\sigma(E)} b_{\sigma(E),j}
= \sum_{A\subseteq E: \sigma(A)=0} (1-1)^{|A|} = 1.\]
The equation $\sum_{i=0}^{\sigma(E)} b_{i,|E|-\sigma(E)} = 1$ can be established by a similar argument or by using duality.
\end{proof}

We now prove our main result on the irreducibility of even delta matroids (see Theorem~\ref{main2}).
\begin{proof}[Proof of Theorem~\ref{main2}]
If $D$ is disconnected then there are non-empty delta-matroids $D_1$ and $D_2$ such that $D=D_1 \oplus D_2$. Proposition~\ref{prop:directsum} implies that $T(D)$ is not irreducible.

Now suppose that $D$ is connected. If $D$ has at most one element then $T(D;x,y)$ is clearly irreducible, so we may assume that $D$ has at least two elements.

Suppose there is a non-trivial factorization
\[ T(D;x,y) = A(x,y)\,C(x,y),\]
where $A(x,y)=\sum_{i,j} a_{i,j}x^iy^j$ and
$C(x,y)=\sum_{i,j} c_{i,j}x^iy^j$. As $D$ has at least one element, we have $a_{0,0}c_{0,0}=b_{0,0}=0$. Without loss of generality, we assume that $a_{0,0}=0$. As $D$ is connected with at least two elements, Theorem~\ref{thm:betaconn} implies that $b_{1,0}\ne 0$. We have $b_{1,0} = a_{1,0}c_{0,0} + a_{0,0}c_{1,0} = a_{1,0}c_{0,0}$, so $c_{0,0}\ne 0$ and $a_{1,0}\neq0$. Similarly $a_{0,1}\neq0$ since $b_{0,1}\ne 0$. We shall obtain a contradiction by proving that $c_{0,0}=0$.

For a polynomial $P(x,y) = \sum_{i,j} p_{i,j}x^i y^j$, define
\begin{align*} \deg_x(P) &= \max\{i: \text{there exists $j$ such that $p_{i,j} \ne 0$}\}\\
\shortintertext{and}
\deg_y(P) &= \max\{j: \text{there exists $i$ such that $p_{i,j} \ne 0$}\}.
\end{align*}
Now let $m(P)= \deg_x(P) + \deg_y(P)$.
Then we have $m(T(D))=m(A)+m(C)$. Furthermore, $\deg_x(A)< m(T(D))$ and $\deg_y(A)<m(T(D))$.

It follows from Lemma~\ref{lem:topcoeffsR} that
\[ 1 = \sum_{j} b_{\deg_x(T(D)),j} = \sum_k a_{\deg_x(A),k}
\sum_l c_{\deg_x(C),l}\]
and similarly, or by duality,
\[ 1 = \sum_{i} b_{i,\deg_y(T(D))} = \sum_k a_{k,\deg_y(A)}
\sum_l c_{l,\deg_y(C)}.\]
Thus,
\begin{equation}\label{eq:topcoeffsA}
\sum_k a_{\deg_x(A),k} \ne 0 \quad \text{and} \quad
\sum_k a_{k,\deg_y(A)} \ne 0.\end{equation}

Now, for $k=0,1,\ldots, m(A)$, let
\[ A_k = \sum_{s=0}^k \sum_{t=0}^{k-s} (-1)^t \binom{k-s}ta_{s,t},\]
and, for $\deg_x(A) \leq k \leq m(A)$ and $i\geq 0$, let
\[ A_{k,i} = \sum_{s=0}^k \sum_{t=0}^{k-s} (-1)^{t+i} \binom{k-s}t a_{s,t+i}.\]
Notice that $A_{k,0}=A_k$. We now prove a recurrence relation involving these quantities.
If $k>\deg_x(A)$ and $i>0$, then, using the identity $\binom{k-s}{t}= \binom{k-s-1}{t}+\binom{k-s-1}{t-1}$, we have
\begin{align*} A_{k,i-1} &=
\sum_{s=0}^{k} \sum_{t=0}^{k-s} (-1)^{t+i-1} \binom{k-s}t a_{s,t+i-1}\\
&= (-1)^{i-1} a_{k,i-1} +
\sum_{s=0}^{k-1} \sum_{t=0}^{k-1-s} (-1)^{t+i-1} \binom{k-1-s}t a_{s,t+i-1}\\
&\phantom{=} {} + \sum_{s=0}^{k-1} \sum_{t=1}^{k-s} (-1)^{t+i-1} \binom{k-1-s}{t-1} a_{s,t+i-1}.\end{align*}
As $k>\deg_x(A)$, we have $a_{k,i-1}=0$. The second and third terms in the equation above are equal to $A_{k-1,i-i}$ and $A_{k-1,i}$ respectively.  Thus,
\begin{equation} A_{k-1,i} = A_{k,i-1} - A_{k-1,i-1} \label{eqn:recurrenceforAij}\end{equation}
whenever $k>\deg_x(A)$ and $i>0$.
By Lemma~\ref{lem:topcoeffsR} and Equation~\eqref{eq:topcoeffsA},
\[ A_{\deg_x(A),\deg_y(A)} = \sum_{s=0}^{\deg_x(A)} (-1)^{\deg_y(A)}a_{s,\deg_y(A)} \ne 0.\]
By applying Equation~\eqref{eqn:recurrenceforAij} repeatedly we can write $A_{\deg_x(A),\deg_y(A)}$ as a linear combination of $A_{\deg_x(A),0}$, \ldots, $A_{m(A),0}$. As $A_{k,0}=A_k$ we see that there exists $k$ with $\deg_x(A) \leq k \leq m(A)$ such that $A_k \ne 0$.

The proof is completed by following exactly the proof of Lemma~4 from~\cite{Merino} to deduce that $c_{0,0}=0$ and thereby reach the desired contradiction.
\end{proof}

\medskip

We conclude by applying Theorem~\ref{main2}  to ribbon graphs.
A ribbon graph $\bG$ is said to be a \emph{join} of ribbon graphs $\bG_1$ and $\bG_2$ if
$\bG$ can be constructed by picking arcs on the boundaries of a non-isolated vertex in each of $\bG_1$ and $\bG_2$ (the arcs should not intersect any edges), then ``gluing together $\bG_1$ and $\bG_2$'' by identifying the arcs and merging the two vertices they lie on into a single vertex of $\bG$.
This definition is consistent with that of the join of embedded graphs given in the introduction.

\begin{theorem}\label{adshjk}
If $\bG$ is an orientable ribbon graph, then $R(\bG;x,y)$ is irreducible over $\mathbb{Z}[x,y]$ (or $\mathbb{C}[x,y]$) if and only if $\bG$ is not a disjoint union or  join of ribbon graphs.
\end{theorem}
\begin{proof}
By Proposition~5.22 of \cite{CMNR-JCT}, the delta-matroid $D(\bG)$ of a ribbon graph $\bG$ is connected if and only if $\bG$ is not a disjoint union or  join of ribbon graphs. By Proposition~5.3 of~\cite{CMNR-JCT}, $\bG$ is  orientable if and only if $D(\bG)$ is even. The result then follows from Theorem~\ref{main2} upon noting that $R(\bG;x,y)=T(D(\bG);x,y)$.
\end{proof}
Theorem~\ref{main} is obtained from Theorem~\ref{adshjk} by rewording it in terms of embedded graphs.

\bibliographystyle{amsplain}
\bibliography{irred.bib}

\providecommand{\bysame}{\leavevmode\hbox to3em{\hrulefill}\thinspace}
\providecommand{\MR}{\relax\ifhmode\unskip\space\fi MR }
\providecommand{\MRhref}[2]{%
  \href{http://www.ams.org/mathscinet-getitem?mr=#1}{#2}
}
\providecommand{\href}[2]{#2}
\begin{thebibliography}{10}

\bibitem{Bobook}
B\'{e}la Bollob\'{a}s, \emph{Modern graph theory}, Graduate Texts in
  Mathematics, vol. 184, Springer-Verlag, New York, 1998.

\bibitem{BR01}
B\'{e}la Bollob\'{a}s and Oliver Riordan, \emph{A polynomial invariant of
  graphs on orientable surfaces}, Proc. London Math. Soc. \textbf{83} (2001),
  no.~3, 513--531.

\bibitem{BR02}
\bysame, \emph{A polynomial of graphs on surfaces}, Math. Ann. \textbf{323}
  (2002), no.~1, 81--96.

\bibitem{Bouchet:symm}
Andr\'{e} Bouchet, \emph{Greedy algorithm and symmetric matroids}, Math.
  Programming \textbf{38} (1987), no.~2, 147--159.

\bibitem{Bouchet:rep}
\bysame, \emph{Representability of {$\triangle$}-matroids}, Combinatorics
  ({E}ger, 1987), Colloq. Math. Soc. J{\'a}nos Bolyai, vol.~52, North-Holland,
  Amsterdam, 1988, pp.~167--182.

\bibitem{Bouchet:maps}
\bysame, \emph{Maps and {$\triangle$}-matroids}, Discrete Math. \textbf{78}
  (1989), no.~1-2, 59--71.

\bibitem{Bouchet:MM3}
\bysame, \emph{Multimatroids. {III}. {T}ightness and fundamental graphs},
  European J. Combin. \textbf{22} (2001), no.~5, 657--677, Combinatorial
  geometries (Luminy, 1999).

\bibitem{Bry}
Thomas~H. Brylawski, \emph{A decomposition for combinatorial geometries},
  Trans. Amer. Math. Soc. \textbf{171} (1972), 235--282.

\bibitem{BO}
Thomas~H. Brylawski and James~G. Oxley, \emph{The {T}utte polynomial and its
  applications}, Matroid applications, Encyclopedia Math. Appl., vol.~40,
  Cambridge Univ. Press, Cambridge, 1992, pp.~123--225.

\bibitem{chand}
Ramaswamy Chandrasekaran and Santosh~N. Kabadi, \emph{Pseudomatroids}, Discrete
  Math. \textbf{71} (1988), no.~3, 205--217.

\bibitem{CCN:splitter}
Carolyn Chun, Deborah Chun, and Steven~D. Noble, \emph{Inductive tools for
  connected delta-matroids and multimatroids}, European J. Combin. \textbf{63}
  (2017), 59--69.

\bibitem{CMNR-JCT}
Carolyn Chun, Iain Moffatt, Steven~D. Noble, and Ralf Rueckriemen,
  \emph{Matroids, delta-matroids and embedded graphs}, J. Combin. Theory Ser. A
  \textbf{167} (2019), 7--59.

\bibitem{CMNR-LMS}
\bysame, \emph{On the interplay between embedded graphs and delta-matroids},
  Proc. London Math. Soc. \textbf{118} (2019), no.~3, 675--700.

\bibitem{Crapo}
Henry~H. Crapo, \emph{The {T}utte polynomial}, Aequationes Math. \textbf{3}
  (1969), 211--229.

\bibitem{dress}
Andreas Dress and Timothy~F. Havel, \emph{Some combinatorial properties of
  discriminants in metric vector spaces}, Adv. in Math. \textbf{62} (1986),
  no.~3, 285--312.

\bibitem{Duchamp}
Alain Duchamp, \emph{Delta matroids whose fundamental graphs are bipartite},
  Linear Algebra Appl. \textbf{160} (1992), 99--112. \MR{1137846}

\bibitem{Fink}
Cl\'{e}ment Dupont, Alex Fink, and Luca Moci, \emph{Universal {T}utte
  characters via combinatorial coalgebras}, Algebr. Comb. \textbf{1} (2018),
  no.~5, 603--651.

\bibitem{EMM}
Joanna~A. Ellis-Monaghan and Criel Merino, \emph{Graph polynomials and their
  applications {I}: {T}he {T}utte polynomial}, Structural analysis of complex
  networks, Birkh{\"a}user/Springer, New York, 2011, pp.~219--255.

\bibitem{EMMbook}
Joanna~A. Ellis-Monaghan and Iain Moffatt, \emph{Graphs on surfaces: Dualities,
  polynomials, and knots}, SpringerBriefs in Mathematics, Springer, New York,
  2013.

\bibitem{handbook}
Joanna~A. Ellis-Monaghan and Iain Moffatt (eds.), \emph{Handbook of the {T}utte
  polynomial and related topics}, CRC Press, 2022.

\bibitem{GLRV}
Andrew Goodall, Bart Litjens, Guus Regts, and Llu\'{\i}s Vena, \emph{A {T}utte
  polynomial for maps {II}: the non-orientable case}, European J. Combin.
  \textbf{86} (2020), 103095, 32.

\bibitem{GKRV}
Andrew~J. Goodall, Thomas Krajewski, Guus Regts, and Llu\'{\i}s Vena, \emph{A
  {T}utte polynomial for maps}, Combin. Probab. Comput. \textbf{27} (2018),
  no.~6, 913--945.

\bibitem{Gordon}
Gary Gordon, \emph{Linear relations for a generalized {T}utte polynomial},
  Electron. J. Combin. \textbf{22} (2015), no.~1, P1.79.

\bibitem{GT87}
Jonathan~L. Gross and Thomas~W. Tucker, \emph{Topological graph theory}, Wiley,
  1987.

\bibitem{HM}
Stephen Huggett and Iain Moffatt, \emph{Types of embedded graphs and their
  {T}utte polynomials}, Math. Proc. Cambridge Philos. Soc. \textbf{169} (2020),
  no.~2, 255--297.

\bibitem{KMT}
Thomas Krajewski, Iain Moffatt, and Adrian Tanasa, \emph{Hopf algebras and
  {T}utte polynomials}, Adv. in Appl. Math. \textbf{95} (2018), 271--330.

\bibitem{Kr}
Vyacheslav Krushkal, \emph{Graphs, links, and duality on surfaces}, Combin.
  Probab. Comput. \textbf{20} (2011), 267--287.

\bibitem{LV}
Michel Las~Vergnas, \emph{Eulerian circuits of 4-valent graphs imbedded in
  surfaces}, Algebraic Methods in Graph Theory (L.~Lov\'asz and V.~S\'os,
  eds.), North Holland, 1981, pp.~451--478.

\bibitem{Merino}
Criel Merino, Anna de~Mier, and Marc Noy, \emph{Irreducibility of the {T}utte
  polynomial of a connected matroid}, J. Combin. Theory Ser. B \textbf{83}
  (2001), no.~2, 298--304.

\bibitem{MMN}
Criel Merino, Iain Moffatt, and Steven~D. Noble, \emph{Series--parallel
  delta-matroids}, in preparation.

\bibitem{Msurvey}
Iain Moffatt, \emph{Delta-matroids for graph theorists}, Surveys in
  combinatorics 2019, London Math. Soc. Lecture Note Ser., vol. 456, Cambridge
  Univ. Press, Cambridge, 2019, pp.~167--220.

\bibitem{MS}
Iain Moffatt and Ben Smith, \emph{Matroidal frameworks for topological {Tutte}
  polynomials}, J. Combin. Theory Ser. B \textbf{133} (2018), 1--31.

\bibitem{MT}
Iain Moffatt and Maya Thompson, \emph{Deletion-contraction and the surface
  {T}utte polynomial}, in preparation.

\bibitem{negami}
Seiya Negami, \emph{Polynomial invariants of embeddings of graphs on closed
  surfaces}, Sci. Rep. Yokohama Nat. Univ. Sect. I Math. Phys. Chem. (1995),
  no.~42, 1--10.

\bibitem{Oxley}
James~G. Oxley, \emph{Matroid theory}, second ed., Oxford Graduate Texts in
  Mathematics, vol.~21, Oxford Univ. Press, Oxford, 2011.

\bibitem{Welsh}
Dominic J.~A. Welsh, \emph{Complexity: knots, colourings and counting}, London
  Mathematical Society Lecture Note Series, vol. 186, Cambridge University
  Press, Cambridge, 1993.

\end{thebibliography}
\end{document}